\documentclass[11pt]{amsart}
\usepackage{geometry}
\usepackage{graphicx}
\usepackage{amssymb}
\usepackage{amsmath, mathrsfs, stmaryrd}
\usepackage{amsthm}
\usepackage{color}

\newtheorem{theorem}{Theorem}

\newcommand{\R}{\mathbb R}
\numberwithin{equation}{section}

\begin{document}

\title{Rigidity and minimizing properties of quasi-local mass}
\author{Po-Ning Chen and Mu-Tao Wang}
\date{\today}

\begin{abstract}
In this article, we survey recent developments in defining the quasi-local mass in general relativity. We discuss various approaches and the properties and applications of the different definitions.  Among the expected properties, we focus on the rigidity property: for a surface in the Minkowski spacetime, one expects that the mass should vanish. We describe the Wang-Yau quasi-local mass whose definition is motivated by this rigidity property and by the Hamilton-Jacobi analysis of the Einstein-Hilbert action. In addition, we survey recent results on the minimizing property the Wang-Yau quasi-local mass.
\end{abstract}

\thanks{P.-N. Chen is supported by NSF grant DMS-1308164 and M.-T. Wang is supported by NSF grants DMS-1105483 and DMS-1405152.} 
\maketitle
\section{Introduction}
One of the major unsolved  problems in general relativity is the definition of conserved quantities such as energy and linear momentum for a finitely extended region in a spacetime \cite{P1}. Many important statements in general relativity make sense only with the presence of a good notion of quasi-local mass. For example, the hoop conjecture predicts the formation of black holes due to the condensation of mass in a finite region. However, there are several difficulties in finding a good notion of energy for a finite region. Firstly, there is no mass density in general relativity due to Einstein's equivalence principle.  Moreover, a generic spacetime does not have symmetry which is important in defining conserved quantities using Noether's principle. As a result, mass can  no longer be defined as a bulk integral of mass density. However, it is conjectured that the mass and other conserved quantities can be defined as integral on the boundary of the region. Such definitions are referred to as quasi-local mass. Moreover, it is expected that a good notion of quasi-local energy should satisfy the following properties  \cite{Christodoulou-Yau,Penrose2}.

\begin{itemize} 
\item Positivity. Quasi-local energy should be positive under the dominant energy condition.
\item Rigidity. The quasi-local mass should be zero for 2-surfaces in the Minkowski spacetime. However, the quasi local energy-momentum could be null (thus quasi-local mass is zero) in a pure radiation spacetime.
\item Asymptotics. The large sphere limits should recover the ADM and Bondi mass in spatial and null infinity, respectively. The small sphere limits should recover matter density in non-vacuum and Bel-Robinson tensor in vacuum.
\item Monotonicity. The quasi-local mass should have some nice monotonicity properties. However, it is not expected that straightforward additivity should hold because the gravitational binding energy could be negative.
\end{itemize}

In comparison to the ADM or Bondi total mass for an isolated system where gravitation is weak at boundary (infinity), the notion of quasi-local mass corresponds to a non-isolated system where gravitation could be strong. 
There are four main approaches to define quasi-local mass based on different expected properties. We give some brief remarks for each approach below.
\begin{itemize} 
\item The variational method. This approach is based on quasi-localizing the ADM mass \cite{Bartnik2}. While the approach guarantees the positivity and monotonicity, it is generally very hard to evaluate such notion of quasi-local mass.
\item Hamilton-Jacobi method. This approach gives the energy as a flux integral on the 2-surface using the Hamilton-Jacobi analysis of the gravitational energy. The surface Hamiltonian is derived in \cite{by1,by2}. See also \cite{hh,ki}.
\item Hawking mass. The Hawking mass \cite{Hawking} is easy to compute and is monotone under the inverse mean curvature flow. However, the positivity does not always hold.
\item Twistor and spinor method. The construction is motivated by the energy-momentum integrals of linearized gravity and is based on twistor ideas. See \cite{Tod} for a survey on the twistor approach to quasi-local mass.
\end{itemize}

Based on their properties, the above definitions of quasi-local mass can be applied to studying different problems in general relativity and differential geometry. For example, the Hawking mass has good asymptotics and monotonicity properties. As a result, it plays a key role in the proof of the Riemannian Penrose inequality of Huisken-Ilmanen. However, the lack of the rigidity property for the Hawking mass hinder the proof of the Penrose inequality at null infinity, as we will describe in later section of this article.  

In this article, we survey some recent developments on quasi-local mass. While we discuss several notions of quasi-local mass and their properties, the main focus is the Wang-Yau quasi-local mass defined in \cite{Wang-Yau1,Wang-Yau2}. This approach is based on Hamilton-Jacobi analysis of energy in general relativity. It is also based on the expectation of the rigidity property that the mass for any surface in the Minkowski spacetime should be zero, which will be the main focus of this survey. 

\section{Riemannian quasi-local energy}
In this section, we survey several definitions of quasi-local mass and describe their main properties.
\subsection{Hawking mass}
By studying the perturbation of the Friedmann-Robertson-Walker (FRW) spacetimes, Hawking defined the following mass which measures the perturbation of the energy from the FRW background \cite{Hawking}.  Let $\Sigma$ be  a spacelike 2-surface in a spacetime $N$ and  $H$ be the mean curvature vector of $\Sigma$ in $N$. The Hawking mass of $\Sigma$ is defined by
\[m_H(\Sigma)=\sqrt{\frac{|\Sigma|}{16\pi}}(1-\frac{1}{16\pi}\int_\Sigma |H|^2 d\Sigma).\]

For a time-symmetric initial data set, the Hawking mass is monotonic increasing under the inverse mean curvature flow. Moreover, for an asymptotically flat initial data, the Hawking mass on large coordinate spheres approaches the ADM mass of the initial data. These properties of the Hawking mass are  instrumental in the  Huisken-Ilmanen's proof of Riemannian Penrose conjecture for a single black hole \cite{HI} which states that the area of the outermost minimal surface gives a lower bound for the ADM mass of a time symmetric initial data.

Despite of the above nice properties, the positivity does not always hold for the Hawking mass, even for surfaces in the Minkowski space. In fact, from the Gauss equation and the Gauss-Bonnet theorem, it is easy to see that for a surface in $\R^3$ with spherical topology, its Hawking mass is always non-positive. Moreover, the Hawking mass is zero if and only if it is the round sphere. 

On the other hand, the Hawking mass is non-negative for stable 2-spheres on time-symmetric hypersurfaces  by Christodoulou and Yau. The positivity is obtained using the positivity of the second variation of the area functional applied to the conformal mapping of the surface to the unit sphere \cite{Christodoulou-Yau}.

\subsection{Bartnik mass}
  Bartnik proposed a quasi-local mass definition by quasi-localizing the ADM mass for asymptotically flat initial data  \cite{Bartnik2}. For simplicity, we focus on the time-symmetric case, namely for 3-manifolds with non-negative scalar curvature.

Let $(\Omega, g)$  be a compact Riemannian 3-manifold of non-negative scalar curvature and with boundary. An admissible extension of $\Omega$ is a complete and asymptotically flat 3-manifold $(\widetilde M,g)$ of non-negative scalar curvature in which $\Omega$ embeds isometrically and is not  enclosed by any compact minimal surfaces.  The Bartnik mass $m_B(\Omega)$ is defined to be the infimum of the ADM mass among all admissible extensions. 

From the definition of Bartnik mass, the non-negativity property is acquired for free, as a consequence of the positive mass theorem.  It is clear that  the Bartnik mass of $(\Omega, g)$ is zero  if the metric $g$ is flat. The converse is also true  as a consequence of the proof of the Riemannian Penrose inequality \cite{HI}. The Bartnik mass is monotonically increasing. Namely, $m_B(\Sigma_2) \ge m_B(\Sigma_1)$ if $\Sigma_2$ contains $\Sigma_1$ isometrically. This follows from the definition since the set of admissible extensions of $\Sigma_2$ is a subset of those of $\Sigma_1$ if $\Sigma_2$ contains $\Sigma_1$.

Despite of the above nice properties, the definition is abstract for both conceptual and computational reasons and one would like to know that the infimum is actually realized by a natural extension of $(\Omega, g)$. This is the content of the static minimization conjecture by Bartnik.

\medskip

\textbf{Static Minimization Conjecture.}  \textit{The infimum
$m_{B}(\Omega)$ is realized by an admissible extension $(\widetilde M,g)$
which is smooth, vacuum, and static outside $\Omega$, is $C^{0,1}$ across
$\partial\Omega$, and has nonnegative scalar curvature across $\partial \Omega$
(in the distributional sense).}

\medskip

See \cite{Andersson-Khuri} for discussion on the  static minimization conjecture, in particular, on the existence and uniqueness of static extension.

\subsection{Brown-York mass}
 
Let $\Sigma$ be a spacelike 2-surface in $N$ whose induced metric has positive Gauss curvature. Assume further that $\Sigma$ bounds a spacelike hypersurface $\Omega$. By the solution of the Weyl’s isometric embedding problem by Nirenberg, there is a unique isometric embedding of $\Sigma$ into $\R^3$. Suppose  $H_0$ is the mean curvature of the isometric embedding of $\Sigma$ into $\R^3$ and $H$ is the mean curvature of $\Sigma$ in $\Omega$. Brown and York define the mass of $\Sigma$ to be
\begin{equation}\label{Brown-York-Mass}
m_{BY}(\Sigma)=\frac{1}{8\pi}\int_\Sigma ( H_0-  H ) \, d \Sigma
\end{equation}
The expression is derived through the Hamilton-Jacobi analysis of Einstein's action with a choice of gauge adopted to the hypersurface $\Omega$.

For time-symmetric hypersurfaces with $k=0$, the dominant energy condition implies that the scalar curvature is non-negative. Shi and Tam proved that the Brown-York mass is positive if the scalar curvature of the enclosed regions is non-negative \cite{Shi-Tam}. Their main idea is to solve the prescribed scalar curvature equation on the exterior of the image of isometric embedding in $\R^3$ using the quasi-spherical method of Bartnik. On the resulting manifold, the Brown-York mass decreases monotonically to the ADM mass at infinity and the positivity of the Brown-York mass follows from the positivity of the ADM mass. While the manifold is not smooth across the boundary of $\Omega$, the proof of the positive mass theorem by Witten  \cite{W} can still be carried out in this setting. 

However,  the Brown-York mass is gauge dependent. Liu and Yau define a mass that is gauge independent by replacing the mean curvature $H$ in equation \eqref{Brown-York-Mass} by the norm of the mean curvature vector of $\Sigma$ in $N$. The Liu-Yau mass is gauge independent and is positive under the dominant energy condition \cite{Liu-Yau}. The proof is based on the idea from the above proof of Shi and Tam and the Jang equation used in the proof of the positive mass theorem by Schoen and Yau \cite{SY2}.

\section{Rigidity property of quasi-local mass}
The main motivation of our investigation is the rigidity property of quasi-local mass. A physically meaningful mass should vanish on any Minkowskian data, as there is no gravitation energy in the Minkowski spacetime.

Though the Liu-Yau mass satisfies the important positivity property, there are 2-surfaces in the Minkowski spacetime
 with strictly positive Liu-Yau mass. In fact, \'{O} Murchadha, Szabados, and Tod \cite{ost} found examples of surfaces in the Minkowski spacetime with  arbitrarily large Liu-Yau mass. These examples are surfaces in the standard light cone of the Minskowski spacetime. For such a surface, the Gauss curvature and the mean curvature vector $\vec{H}$ satisfies the following identity
\[ 4K = |\vec{H}|^2. \]
Using the Gauss equation for surfaces in $\R^3$, it is not hard to see that the Liu-Yau mass can be arbitrarily large for such surfaces. The missing of the momentum information is responsible for this inconsistency.  For the Liu-Yau mass, the reference is taken to be the isometric embedding into $\R^3$ which is a totally geodesics hypersurface. In order to capture the information of the second fundamental form, we need to take the reference surface to be a general isometric embedding into the Minkowski spacetime.

The Hawking mass is always negative for a 2-surface in $\R^3$ unless it is a metric round sphere. On the other hand, the Hawking mass is always zero for surfaces in the standard light cone in $\R^{3,1}$. Moreover, for surfaces in the standard light cone in the Schwarzschild spacetime, the Hawking mass is always greater than the total mass of the Schwarzschild spacetime, unless the surface also lies in the static hypersurface. Moreover, the Hawking mass can be arbitrarily large. Such a drawback hinders the proof of the Penrose inequality at null infinity using the Hawking mass, even though there is a corresponding monotonicity formula along the null, see \cite{Sauter}.

There is a similar phenomenon for the Hawking mass for surfaces in asymptotically Anti-de-Sitter (Ads) spacetimes. Namely, for surfaces approaching the infinity, the Hawing mass can be arbitrarily large depending on the shape of the surface. While this shows the lack of rigidity for the Hawking mass, this "drawback" of the Hawking mass is used as a tool to prove insufficient convergence of inverse mean curvature flow in asymptotically AdS spacetimes \cite{Neves}.
\section{Wang-Yau quasi-local energy}
In this section, we review the definition of the quasi-local energy-momentum in  \cite{Wang-Yau1,Wang-Yau2}. The main motivation of this definition is the rigidity property that surfaces in the Minkowski spacetime should have zero mass. As a result, all possible isometric embeddings $X$ of the surface into $\R^{3,1}$ are used  as references and an energy is assigned to each pair $(X, T_0)$ of an isometric embedding $X$ and a unit timelike vector $T_0$ in $\R^{3,1}$.

Let $\Sigma$ be a closed embedded spacelike 2-surface in a spacetime with spacelike mean curvature vector $H$. The data used in the definition of the Wang-Yau quasi-local mass is the triple $(\sigma,|H|,\alpha_H)$ where $\sigma$ is the induced metric on $\Sigma$, $|H|$ is the norm of the mean curvature vector and $\alpha_H$ is the connection one form of the normal bundle with respect to the mean curvature vector
\[ \alpha_H(\cdot )=\langle \nabla^N_{(\cdot)}   \frac{J}{|H|}, \frac{H}{|H|}  \rangle  \]
where $J$ is the reflection of $H$ through the incoming  light cone in the normal bundle.

Given an isometric embedding $X:\Sigma\rightarrow \R^{3,1}$ and a constant future timelike unit vector $T_0\in \R^{3,1}$, we consider the projected embedding $\widehat{X}$ into the orthogonal complement of $T_0$. We denote the induced metric, the second fundamental form, and the mean curvature of the image by $\hat{\sigma}_{ab}$, $\hat{h}_{ab}$, and $\widehat{H}$, respectively.  The Wang-Yau quasi-local energy with respect to $(X, T_0)$ is 
\[\begin{split}&E(\Sigma, X, T_0)=\int_{\widehat{\Sigma}}\widehat{H}d{\widehat{\Sigma}}-\int_\Sigma \left[\sqrt{1+|\nabla\tau|^2}\cosh\theta|{H}|-\nabla\tau\cdot \nabla \theta -\alpha_H ( \nabla \tau) \right]d\Sigma,\end{split}\] where $\theta=\sinh^{-1}(\frac{-\Delta\tau}{|H|\sqrt{1+|\nabla\tau|^2}})$, $\nabla$ and $\Delta$ are the gradient and Laplacian, respectively, with respect to $\sigma$ and $\tau= - X \cdot T_0$ is the time function.

In \cite{Wang-Yau1,Wang-Yau2}, it is proved that, $E(\Sigma, X, T_0) \ge 0$ if $\Sigma$ bounds a spacelike hypersurface, the dominant energy condition holds in $N$ and the pair $(X,T_0)$ is admissible. The Wang-Yau quasi-local mass is defined to be the minimum  of the quasi-local energy $E(\Sigma, X, T_0)$ among all admissible pairs $(X, T_0)$. In particular, for a surface in the Minkowski spacetime, its Wang-Yau mass is zero. However, for surfaces  in a generic spacetime, it is not clear which isometric embedding would minimize the quasi-local energy. To find the isometric embedding that minimizes the quasi-local energy, we study the Euler-Lagrange equation for the critical point of the Wang-Yau energy as a functional of $\tau$. It is the following fourth order nonlinear elliptic equation

\[
 -(\widehat{H}\hat{\sigma}^{ab} -\hat{\sigma}^{ac} \hat{\sigma}^{bd} \hat{h}_{cd})\frac{\nabla_b\nabla_a \tau}{\sqrt{1+|\nabla\tau|^2}}+ div_\sigma (\frac{\nabla\tau}{\sqrt{1+|\nabla\tau|^2}} \cosh\theta|{H}|-\nabla\theta-\alpha_{H})=0
\]
coupled with the isometric embedding equation for $X$. It is referred to as the optimal isometric embedding equation.

The data for the image of the isometric embedding in the Minkowski spacetime can be used to simplify the expression for the quasi-local energy and the optimal isometric embedding equation. Denote the norm of the mean curvature vector and the connection one-form in mean curvature gauge of the image surface of $X$ in $\R^{3,1}$ by $|H_0|$ and $\alpha_{H_0}$, respectively. We have the following identities relating the geometry of the image of the isometric embedding $X$ and 
the image surface $\widehat{\Sigma}$ of $\widehat{X}$  \cite{Chen-Wang-Yau2}.
\[\sqrt{1+|\nabla\tau|^2}\widehat{H} =\sqrt{1+|\nabla\tau|^2}\cosh\theta_0|{H_0}|-\nabla\tau\cdot \nabla \theta_0 -\alpha_{H_0} ( \nabla \tau) \]
\[   -(\widehat{H}\hat{\sigma}^{ab} -\hat{\sigma}^{ac} \hat{\sigma}^{bd} \hat{h}_{cd})\frac{\nabla_b\nabla_a \tau}{\sqrt{1+|\nabla\tau|^2}}+ div_\sigma (\frac{\nabla\tau}{\sqrt{1+|\nabla\tau|^2}} \cosh\theta_0|{H_0}|-\nabla\theta_0-\alpha_{H_0})=0.
\]

The second identity simply states that a surface inside $\R^{3,1}$ is a critical point of the quasi-local energy with respect to other isometric embeddings back to $\R^{3,1}$. This can be proved by either the positivity of energy or a direct computation. We  substitute these relations into the expression for $E(\Sigma, X, T_0)$ and the optimal isometric embedding equation, and rewrite them in term of

\[
\begin{split}
 \label{rho}\rho &= \frac{\sqrt{|H_0|^2 +\frac{(\Delta \tau)^2}{1+ |\nabla \tau|^2}} - \sqrt{|H|^2 +\frac{(\Delta \tau)^2}{1+ |\nabla \tau|^2}} }{ \sqrt{1+ |\nabla \tau|^2}}
\\ 
j_a & =\rho {\nabla_a \tau }- \nabla_a [ \sinh^{-1} (\frac{\rho\Delta \tau }{|H_0||H|})]-(\alpha_{H_0})_a + (\alpha_{H})_a.\end{split}
\]

In terms of these, the quasi-local energy is $\frac{1}{8\pi}\int_\Sigma (\rho+j_a\nabla^a\tau)$ and a pair of an embedding $X:\Sigma\hookrightarrow \mathbb{R}^{3,1}$ and an observer $T_0$ satisfies the optimal isometric embedding equation if $X$ is an isometric embedding and 
\[
div_{\sigma} j=0.
\] 
\section{Minimizing properties of the Wang-Yau quasi-local mass}
In this section, we discuss the minimizing properties of the Wang-Yau quasi-local mass. In general, the optimal embedding equation is a nonlinear system and is rather difficult to solve. However, $\tau=0$ is a solution of the optimal embedding equation if $div \alpha_H=0$. This special case was studied by Miao--Tam--Xie  \cite{MTX}. They estimate the second variation of quasi-local energy around the critical point $\tau=0$  by linearizing the optimal embedding equation nearby. It is shown that the second variation is 
\[
\frac{1}{8 \pi} \int_{\Sigma} \frac{(\Delta f)^2}{H} + (H_0 - |H|) |\nabla f|^2 - \hat h(\nabla f, \nabla f).
\]
Assuming $H_0 > |H|$, it is shown that the integral is positive using a generalization of Reilly's formula. As a result, they obtained sufficient conditions for the critical point $\tau=0$ to be a local minimum. 

To study the minimizing property of a general solution of the optimal isometric embedding equation, a different method has to be devised to deal with the fully nonlinear nature of the equation. In \cite{Chen-Wang-Yau2}, we proved the following result.
\begin{theorem} \label{thm_local}
Suppose  $\Sigma$ is a spacelike 2-surface and $\tau_0$ is a critical point of the quasi-local energy functional $E(\Sigma, \tau)$. Assume further that 
\[ |H_{\tau_0}| > |H| >0\]
where $H_{\tau_0}$ is the mean curvature vector of the isometric embedding $X_{\tau_0}$ of $\Sigma$ into $\R^{3,1}$ with time function $\tau_0$.
Then, 
 $\tau_0$ is a local minimum for  $E(\Sigma,\tau)$.
\end{theorem}
The result is based on the following comparison theorem of quasi-local energy.
\begin{theorem} \label{thm_int}
Suppose  $\Sigma$ is a spacelike 2-surface and $\tau_0$ is a critical point of the quasi-local energy functional $E(\Sigma,\tau)$.  
Assume further that 
\[ |H_{\tau_0}| > |H|>0. \]
Then,  for any time function $\tau$ such that $ \sigma + d \tau \otimes d \tau$ has positive Gaussian curvature, we have
\begin{equation}\label{comparison} E(\Sigma,\tau) \ge E(\Sigma,\tau_0) +E(\Sigma_{\tau_0}, \tau).\end{equation}
Moreover, equality holds if and only if $\tau-\tau_0$ is a constant.  
\end{theorem}
In the second term on the right hand side of \eqref{comparison}, we view $\Sigma_{\tau_0}$  in the Minkowski spacetime as a physical surface and consider the isometric embedding with time function $\tau$ as reference.

In view of Theorem \ref{thm_int}, it suffices to show that $E(\Sigma_{\tau_0},\tau )$ is non-negative for $\tau$ close to $\tau_0$ in order to prove Theorem \ref{thm_local}. In \cite{Chen-Wang-Yau2}, this is achieved by verifying the admissible condition directly. In the following, we present a new result on the second variation of $E(\Sigma_{\tau_0},\tau )$ around $\tau_0$.

\begin{theorem}
Consider a standard coordinate system $(x_0,x_1,x_2,x_3)$ for the Minkowski spacetime. Suppose $\Sigma$ is a spacelike surface in the Minkowski spacetime with 
embedding $X$. Assume further that the time function is $\tau$ and the projection of $\Sigma$ onto $\R^3$ is convex. Then, for any function $f$  on $\Sigma$, we have
\[
\begin{split} 
E(\Sigma,\tau)=&0\\
\partial_s E(\Sigma,\tau+sf)|_{s=0}=&0\\
\partial^2_s E(\Sigma,\tau+sf)|_{s=0}\ge&0\\
\end{split}
\]
Moreover, equality holds in the last inequality if and only if $f$ is the restriction of $a_0+ \sum a_i x_i$ to $\Sigma$ for some constants $a_0$ and $a_i$.
\end{theorem}
\begin{proof}
These results, except the equality cases, are proved in Theorem 2 and Lemma 1 of \cite{Chen-Wang-Yau2} where the positivity of $ E(\Sigma,\tau+sf)$ is obtained by the positivity of the Wang-Yau mass. In the following, we go over the proof of the positivity of quasi-local energy and extract the necessary information for the equality case.

Let $\Omega$ be a spacelike hypersuface with boundary $\Sigma$. Let $g_{ij}$ and $p_{ij}$ be the induced metric and the second fundamental form of $\Omega$ in $\R^{3,1}$, respectively. 
Let $e_3$ be the unit outward normal of $\Sigma$ in $\Omega$ and $e_4$ be the future directed 
unit normal of $\Omega$ in $\R^{3,1}$. The Jang equation is used in both the proof of positive mass theorem in \cite{SY2} and the proof of the positivity of quasi-local energy in \cite{Wang-Yau2}. We will use it here to study the equality case for the second variation. In a local coordinate on $\Omega$, the Jang equation takes the following form: 
\[   \sum_{i,j=1}^3(g^{ij }   -\frac{f^i f^j}{1+|Df|^2}) (\frac{D_iD_j f}{\sqrt{1+ |Df|^2}} -p_{ij}) =0  \]
where $D$ is the covariant derivative with respect to the metric $g_{ij}$.

Let $F$ be the solution to Dirichlet problem of Jang's equation on $\Omega$ with boundary value $\tau$. Consider a variation of the time function 
by $\delta \tau$. Let $F+ \delta F$ be the solution to Dirichlet problem of Jang's equation on $\Omega$ with boundary value $\tau+\delta \tau$. 
Let $\widetilde \Omega_{F+ \delta F}$ be the graph of $F+ \delta F$  over $\Omega$ in $\Omega \times \R$. Let $R_{F+ \delta F}$ be the scalar curvature of  $\widetilde \Omega_{F+ \delta F}$. 
By \cite{SY2}, there exists a vector field $V_{F+ \delta F}$ on $\widetilde \Omega_{F+ \delta F}$  such that 
\begin{equation} \label{scalar_inequality}
R_{F+ \delta F}+ 2 div (V_{F+ \delta F}) -2|V_{F+ \delta F}|^2 \ge \sum_{i, j}(h_{ij}-p_{ij})^2 \end{equation}
where $h_{ij}$ is the second fundamental form of $\widetilde \Omega_{F+ \delta F}$ in $\Omega \times \R$. Indeed, 
\[  h_{ij}= \frac{D_iD_j f}{\sqrt{1+ |Df|^2}}.\]

By Proposition 2.1 of \cite{Wang-Yau2},
\[ E(\Sigma, \tau + \delta \tau) \ge  \int_{\widehat \Sigma_{ \tau + \delta \tau}} \widehat H_ {\tau + \delta \tau} -\int_{\Sigma} h(\Sigma , \tau , e_3'),  \]
where 
\[ e'_3= \cosh \theta e_3 + \sinh \theta e_4   \]
with
\[  \sinh \theta = \frac{-e_3(F+ \delta F)}{\sqrt{1+|\nabla (\tau + \delta \tau)|^2}} \]
and  $h(\Sigma , \tau , e_3')$ is the generalized mean curvature with respect to $e'_3$.

Consider the exterior region $M_1$ of $\widehat \Sigma_{ \tau + \delta \tau}$ in $\R^3$, which is foliated by the level set of the distance function to the surface. Using the foliation, one can rewrite the flat metric of $\R^3$ as 
\[ g= dr^2+ g(r,u^a)_{ab} du^adu^b.\]
Consider a lapse function $u$ and new metric of the form
\[ g'=u^2dr^2+ g(r,u^a)_{ab} du^adu^b \]
on $M_1$. Consider the equation 
\[R(g') =0 \]
with boundary condition 
\[ u =\frac{\widehat H_{ \tau + \delta \tau}}{\tilde H- \langle V_{F+ \delta F}, \nu \rangle}  \]
where $\tilde H$ and $\nu$ are the mean curvature  and unit normal vector of $\widetilde \Sigma_{\tau + \delta \tau}$ in $\widetilde \Omega_{F+ \delta F}$, respectively. 
The solution gives an asymptotically flat metric on $M_1$ such that
\[  \int_{\widehat \Sigma_{ \tau + \delta \tau}} \widehat H_ {\tau + \delta \tau} -\int_{\Sigma} h(\Sigma , \tau , e_3')  \ge  m_{ADM}(M_1,g').\] 
To show the positivity of $ m_{ADM}(M_1,g')$, consider the manifold $M$ obtained by glueing  $\widetilde \Omega_{F+ \delta F}$ and $M_1$. 
Let $\mathcal D$ be the Dirac operator and $\nabla_s$ be the spin connection of the spinor bundle on $M$. Let $\Psi$ be the solution to the Dirac equation 
\[ \mathcal D \Psi =0,  \]
which approaches a constant spinor at infinity. Then, by section 5 of \cite{Wang-Yau2},
\[ m_{ADM}(M_1)= \int_{M_1}  |\nabla_s \Psi|^2 + \int_{\widetilde \Omega_{F+ \delta F}}  |\nabla_s \Psi|^2 + \frac{1}{4} R_{F+ \delta F} |\Psi|^2.\]
Moreover, for any vector field $Y$ on $ \Omega_{F+ \delta F} $,
\[ \int_{\widetilde \Omega_{F+ \delta F}}  |\nabla_s \Psi|^2 + \frac{1}{4} R_{F+ \delta F} |\Psi|^2 \ge \int_{\widetilde \Omega_{F+ \delta F}}  \frac{1}{2}|\nabla_s \Psi|^2 + \frac{1}{4} (R_{F+ \delta F}+2 div Y -2|Y|^2 )|\Psi|^2.  \]
If we apply the equation  to $Y=V_{F+\delta F}$ and use equation \eqref{scalar_inequality}, it follows that
\[   \int_{\widetilde \Omega_{F+ \delta F}}  |\nabla_s \Psi|^2 + \frac{1}{4} R_{F+ \delta F} |\Psi|^2   \ge \int_{\widetilde \Omega_{F+ \delta F}}  \frac{1}{2}|\nabla_s \Psi|^2 + \frac{1}{4} \sum_{i, j}(h_{ij}-p_{ij})^2|\Psi|^2.\]

Combining these results, we have
\[ E(\Sigma, \tau+\delta \tau) \ge  \int_{\widetilde \Omega_{F+ \delta F}}  \frac{1}{2}|\nabla_s \Psi|^2+ \frac{1}{4} \sum_{ij}(h_{ij}-p_{ij})^2|\Psi|^2.\]
When $\delta \tau =0$, we have
\[ E(\Sigma, \tau)=0. \]
Indeed, the Dirac spinor $\Psi$ in this case is a constant spinor, $\nabla_s \Psi=0$, and $h_{ij}=p_{ij}$.
As a result, one has the following lower bound  for the second variation
\[   \int_{\widetilde \Omega_{F+ \delta F}}    \sum_{ij} (\delta h_{ij})^2|\Psi_0|^2   \]
where $\Psi_0$ is the constant spinor obtained when $\delta \tau =0$. Hence, the second variation is strictly positive unless $\delta h_{ij}=0$. On the other hand,
\[   \delta h_{ij}= \frac{D_i D_j  \delta F}{\sqrt{1+|D F|^2}}- \frac{D_i D_j F}{(1+|D F|^2)^{\frac{3}{2}}} D F \cdot D \delta F.  \]
As a result, the second variation is positive unless $\delta F$ satisfies
\[    \frac{D_i D_j  \delta F}{\sqrt{1+|D F|^2}}- \frac{D_i D_j F}{(1+|D F|^2)^{\frac{3}{2}}} D F \cdot D \delta F =0.\]
This is the same as 
\begin{equation} \label{Hessian}
D_i D_j  \delta F = p_{ij} \frac{ D F \cdot D \delta F}{\sqrt{1+|D F|^2}}. 
\end{equation}
However, the covariant derivative, $D$, on hypersurface $\Omega$ and the covariant derivative of $\R^{3,1}$, $\nabla^{(4)}$, are related by
 \[ D_{i} D _j f= \nabla_i^{(4)} \nabla_j^{(4)} f - p_{ij}e_4(f)\]
for any function $f$ on  $\R^{3,1}$.
As a result, if $\delta F$ satisfies equation \eqref{Hessian}, we can extend $\delta F$ suitably such that 
\[  \nabla_i^{(4)} \nabla_j^{(4)} \delta F=0.\]
This shows that, 
\[  \delta F = (a_0 + \sum_i a_i X^i)|_{\Omega}. \]
Thus, the second variation for $E(\Sigma, \tau+ \delta \tau)$ being zero implies that 
$\delta \tau$ is the restriction of $ a_0 + \sum_i a_i X^i $ to $\Sigma$. 

On the other hand, it is easy to show that if $\delta \tau$ is the restriction of $ a_0 + \sum_i a_i X^i $ to $\Sigma$ then the second variation is zero. Indeed, such $\delta \tau$  corresponds to rotating the observer when the data remains unchanged. As a result, the energy remains zero. In particular, the second variation is zero.
\end{proof}

\end{document}